\documentclass[11pt,reqno]{amsart}
\usepackage{amscd}
\usepackage{amssymb}
\usepackage{hyperref}

\hypersetup{pdftitle={SO(3)-monopoles: The overlap problem}}
\hypersetup{pdfauthor={Paul M. N. Feehan and Thomas G. Leness}}

\renewcommand\emptyset{\varnothing}

\newcommand\barM{{\bar{M}}}

\newcommand\barmu{{\bar\mu}}

\newcommand\AAA{\mathbb{A}}

\newcommand\CC{\mathbb{C}}

\newcommand\QQ{\mathbb{Q}}
\newcommand\RR{\mathbb{R}}

\newcommand\ZZ{\mathbb{Z}}

\newcommand\beps{{\boldsymbol{\varepsilon}}}

\newcommand\bga{{\boldsymbol{\gamma}}}

\newcommand\bA{{\mathbf{A}}}

\newcommand\bB{{\mathbf{B}}}

\newcommand\bL{{\mathbf{L}}}

\newcommand\bx{{\mathbf{x}}}

\newcommand{\rd}{\partial}

\newcommand\fg{{\mathfrak{g}}}

\newcommand\fo{{\mathfrak{o}}}

\newcommand\fp{{\mathfrak{p}}}

\newcommand\fs{{\mathfrak{s}}}
\newcommand\fS{{\mathfrak{S}}}
\newcommand\ft{{\mathfrak{t}}}

\newcommand\eps{\varepsilon}
\newcommand\ga{\gamma}
\newcommand\Ga{\Gamma}
\newcommand\la{\lambda}

\newcommand\ka{\kappa}
\newcommand\om{\omega}
\newcommand\Om{\Omega}
\newcommand\si{\sigma}
\newcommand\Si{\Sigma}

\newcommand\su{{\mathfrak{s}\mathfrak{u}}}

\newcommand\BG{\operatorname{BG}}

\newcommand\EG{\operatorname{EG}}

\newcommand\SO{\operatorname{SO}}

\newcommand\SU{\operatorname{SU}}

\newcommand{\8}{\infty}

\newcommand\asd{{\operatorname{asd}}}

\newcommand\CCl{\operatorname{{\mathbb{C}\ell}}}
\newcommand\Char{\operatorname{Char}}

\newcommand\End{\operatorname{End}}

\newcommand\Gl{\operatorname{Gl}}

\newcommand\ind{\operatorname{Index}}

\newcommand\Imag{\operatorname{Im}}

\newcommand\Red{\operatorname{Red}}

\newcommand\SW{SW}
\newcommand\Sym{\operatorname{Sym}}

\newcommand\vol{\operatorname{vol}}

\newcommand\cl{{\mathrm{cl}}}

\newcommand\id{{\mathrm{id}}}

\newcommand\spinc{\text{$\text{spin}^c$ }}
\newcommand\spinu{\text{$\text{spin}^u$ }}
\newcommand\Spinc{\text{$\text{Spin}^c$}}
\newcommand\Spinu{\text{$\text{Spin}^u$}}
\newcommand\vir{\text{vir}}

\newcommand\sA{{\mathcal{A}}}
\newcommand\sB{{\mathcal{B}}}
\newcommand\sC{{\mathcal{C}}}

\newcommand\sG{{\mathcal{G}}}

\newcommand\sM{{\mathcal{M}}}
\newcommand\sN{{\mathcal{N}}}
\newcommand\sO{{\mathcal{O}}}

\newcommand\sV{{\mathcal{V}}}
\newcommand\sW{{\mathcal{W}}}

\newcommand\tsC{{\tilde\sC}}

\newtheorem{thm}{Theorem}[section]
\newtheorem{prop}{Proposition}[section]

\theoremstyle{definition}
\newtheorem{definition}[thm]{Definition}

\theoremstyle{remark}
\newtheorem{rmk}[thm]{Remark}

\numberwithin{equation}{section}

\begin{document}

\title[Overlap problem]
{SO(3)-monopoles: The overlap problem}

\author[Paul M. N. Feehan]{Paul M. N. Feehan}
\address{Department of Mathematics\\
        Rutgers University\\
        Piscataway, NJ 08854-8019}
\email{feehan@math.rutgers.edu}
\urladdr{http://www.math.rutgers.edu/$\sim$feehan}

\author[Thomas G. Leness]{Thomas G. Leness}
\address{Department of Mathematics\\
        Florida International University\\
        Miami, FL 33199} \email{lenesst@fiu.edu}
\urladdr{http://www.fiu.edu/$\sim$lenesst}
%\dedicatory{}
%PF 11.2.2012 \subjclass{Primary 58D27; Secondary 57R57, 58Gxx, 53Cxx, 58D29, 57R15, 58Cxx}
\subjclass{Primary 57R57; Secondary 58D27 58D29}
\thanks{PMNF was supported in part by NSF grants DMS-9704174 and DMS-0125170. TGL was supported in part by NSF grant DMS-0103677.}
\date{May 7, 2005}
%PF 11.2.2012 \date{This version: \today. math.DG/yymmnnn.}
\keywords{}
\begin{abstract}
The $\SO(3)$-monopole program, initiated by Pidstrigatch and Tyurin
\cite{PTLocal}, yields a relationship between the Donaldson and
Seiberg-Witten invariants through a cobordism between the moduli
spaces defining these invariants. The main technical difficulty in
this program lies in describing the links of singularities in this
cobordism arising from the Seiberg-Witten moduli subspaces. In
\cite{FL3}, we defined maps which, essentially, define normal
bundles of strata of these singularities. The link in question is
then the boundary of the union of the tubular neighborhoods
associated with these normal bundles.  However, the
$\SO(3)$-monopole program requires the computation of intersection
numbers with links where more than one stratum appears in the family
of singularities and thus more than one tubular neighborhood appears
in the definition of the link.  Computations of intersection numbers
in unions of open sets have proved difficult for even two open sets,
\cite{OzsvathBlowUp,LenessWC}. In this note, we give a brief
introduction to our article \cite{FL5}, in which we implement these
computations.
\end{abstract}
\maketitle

\section{Introduction}
Before the introduction of Seiberg-Witten invariants \cite{Witten},
the Donaldson invariants were the chief means of distinguishing
between smooth structures on four-manifolds. In \cite{Witten},
Witten not only defined Seiberg-Witten invariants, which are easier
to compute, but also conjectured a relation between the Donaldson
and Seiberg-Witten invariants. Assuming the conjecture, one see that
the Donaldson and Seiberg-Witten invariants contain the same
information about the smooth structure of a four-manifold.

In \cite{PTLocal}, Pidstrigatch and Tyurin introduced the
$\SO(3)$-monopole program to prove Witten's conjecture; see also an
account by Okonek and Teleman in \cite{OTQuaternion,OTVortex}. We
provide a description of this program in \cite{FLGeorgia}.

In \cite{FL1,FL2a,FL2b,FLLevelOne}, we proved that for an
appropriate choice of a \spinu structure $\ft$ (defined in \S
\ref{subsec:SpinuStr}), an Uhlenbeck-type compactification of the
moduli space of $\SO(3)$ monopoles, $\bar\sM_{\ft}/S^1$, defines a
smoothly-stratified cobordism between a link of a moduli space of
anti-self-dual connections and links of singularities of the form
$M_{\fs}\times\Sym^\ell(X)$ where $M_{\fs}$ is the Seiberg-Witten
moduli space associated to the \spinc structure $\fs$ (see
\cite{MorganSWNotes}).  In this note, we give an introduction to the
ideas underlying the proof of our main result in \cite{FL5}.
This result relies on a technical result, stated here as
Theorem \ref{thm:ExtendedGluingThm}, the proof of which should
be a routine extension of the results of \cite{FL3}, and which
will appear in \cite{FL4}.

\begin{thm}
\label{thm:CobordismResult} Let $X$ be a smooth, oriented manifold
with $b^1(X)=0$. Let $\Char(X)\subset H^2(X;\ZZ)$ be the set of
integral lifts of $w_2(X)$.   Then, for $h\in H_2(X;\RR)$,
$w\in H^2(X;\ZZ)$, and generator $x\in H_0(X;\ZZ)$,
\begin{equation}
\label{eq:MainThm}
D^w_X(h^{\delta-2m}x^m)
=
-\sum_{c\in \Char(X)}
\SW_X(c)g^w_{X,\delta,m,c}(h^{\delta-2m}x^m),
\end{equation}
where $D^w_X(h^{\delta-2m}x^m)$ is the Donaldson invariant of $X$,
$\SW_X(c)$ is the Seiberg-Witten invariant of the \spinc structure
$\fs$ with $c_1(\fs)=c$, and
$$
g^w_{X,\delta,m,c}:
\Sym(H_0(X;\ZZ)\oplus H_2(X;\QQ)) \to \QQ
$$
is a universal function depending only on
$\delta$, $m$, $w$, $c$ and the homotopy type of $X$.
\end{thm}

The proof of Theorem \ref{thm:CobordismResult} proceeds as follows.
On a dense, open subspace
$$
\sM^{*,0}_{\ft}/S^1 \subset \bar\sM_{\ft}/S^1,
$$
there are geometric representatives $\bar\sV(h^{\delta-2m}x^m)$ and
$\bar\sW$ which can be thought of, following  \cite[\S
2(ii)]{KMStructure}, as representatives of homology classes. The
geometric representative $\bar\sV(h^{\delta-2m}x^m)$ is essentially
that defined in \cite[\S 2(ii)]{KMStructure}, where it is used to
compute the Donaldson invariant $D^w_X(h^{\delta-2m}x^m)$. The
geometric representative $\bar\sW$ is a representative of the
homology class Poincar\'e dual to a multiple of the first Chern
class of the $S^1$ action on $\sM^{*,0}_{\ft}$. The intersection,
$$
\sM^{*,0}_{\ft}/S^1\cap\bar\sV(h^{\delta-2m}x^m)\cap\bar\sW^{n-1},
$$
where $n=n_a(\ft)$ is a non-negative integer  defined in
equation \eqref{eq:ModuliDim}, is a collection of oriented,
one-dimensional manifolds. The cobordism given by these
one-dimensional manifolds yields an identity:
\begin{equation}
\label{eq:CobordismSum}
2^{n-1}D^w_X(h^{\delta-2m}x^m)
=
-\sum_{\fs\in\Spinc(X)}
\#\left(\bar\sV(h^{\delta-2m}x^m)\cap\bar\sW^{n-1}\cap\bL_{\ft,\fs}\right),
\end{equation}
where $\bL_{\ft,\fs}$ is  the link of $M_{\fs}\times\Sym^\ell(X)$ in
$\bar\sM_{\ft}/S^1$. Theorem \ref{thm:CobordismResult}  then follows
from a partial computation of the intersection numbers in
\eqref{eq:CobordismSum}:
\begin{equation}
\label{eq:ReducibleLink}
\#\left(\bar\sV(h^{\delta-2m}x^m)\cap\bar\sW^{n-1}\cap\bL_{\ft,\fs}\right)
=\SW_X(\fs)g^w_{X,\delta,m,c_1(\fs)}(h^{\delta-2m}x^m),
\end{equation}
where $g^w_{X,\delta,m,c_1(\fs)}$ is as defined in
Theorem \ref{thm:CobordismResult} and $c_1(\fs)$ is
the characteristic class defined in \eqref{eq:DefineChernClassOfSpinc}.
A more detailed statement of \eqref{eq:ReducibleLink} appears
in Theorem \ref{conj:PTConj}.

The proofs of Equation \eqref{eq:ReducibleLink} for $\ell=0$ and
$\ell=1$ appear in \cite{FL2b} and \cite{FLLevelOne} respectively.
Both of these proofs proceed by presenting the link $\bL_{\ft,\fs}$
as the intersection of the zero-locus of an obstruction bundle with
the boundary of a disk bundle (or an orbifold disk bundle when
$\ell=1$), showing that the intersection number is given by a
cohomological pairing, and then computing the cohomology ring of
this disk bundle. For $\ell>1$, the link $\bL_{\ft,\fs}$ is the
intersection of the zero-locus of an obstruction bundle with the
boundary of a {\em union\/} of cone bundles.  By a cone bundle, we
mean a fiber bundle whose fiber is given by a cone on a space which
need not be a sphere. The topology of these cones is sufficiently
complicated that the implied computation in Equation
\eqref{eq:ReducibleLink} can only be partial and not as explicit as
one might like. A more profound difficulty is that, as in all
Mayer-Vietoris arguments, if one wants to compute the cohomology of
a union, one must understand the intersection of the elements of
that union.  We refer to this problem as the {\em overlap
problem\/}.  Our goal in this note is to provide an exposition of
the lengthy proof \cite{FL5} of Equation \eqref{eq:ReducibleLink},
in which we solve the overlap problem.

\subsection{A model for intersection theory on stratified links}
\label{subsec:Model}
As noted above, the link $\bL_{\ft,\fs}$ is a subspace of
the boundary of a union of cone bundles and to compute the intersection number
\eqref{eq:ReducibleLink}, we must understand the intersections of
these cone bundles. A precise accounting of what information about such a
neighborhood is necessary for this computation can be given more concisely in
an abstract language which we now introduce.

Let $S\subset Y$ be a closed subspace of a smoothly stratified
space $Y$ by which we mean that $Y$ is the disjoint union of
smooth manifolds $Y_i$, which we call {\em strata\/}, where $i$ lies in a partially
ordered set satisfying $i<j$ if and only if $Y_i\subseteq
\cl(Y_j)$. We assume that $Y$ is the closure of the highest
stratum $Y_n$ and refer to $\cup_{i<n}Y_i$ as the {\em lower
strata\/}. In addition, we assume the strata of $Y$ satisfy the
{\em condition of frontier\/} that $Y_i\cap\cl(Y_j)\neq\emptyset$ only if
$Y_i\subseteq \cl(Y_j)$. We will write $S_i=S\cap Y_i$ and assume
that $S_i$ is a smooth submanifold of $Y_i$.

If $S$ were a submanifold of a smooth manifold, the standard language
for computing intersection numbers with the link of $S$ would be
to introduce a normal bundle $N\to S$ and an embedding $\bga:\sO\subset N\to Y$
where $\sO$ is a closed disk subbundle.  The link is then defined as
$\bL=\rd(\bga(\sO))$.

Neighborhoods of stratified subspaces need not be so simple. Because
stratified spaces are not manifolds, neighborhoods of subspaces
cannot necessarily be parameterized by vector bundles. Instead of
vector bundles, such neighborhoods are parameterized by cone bundles
with fiber given by a cone on a more complicated topological space.
Much of the following description applies to a wide variety of such
spaces, including Whitney or Thom-Mather stratified spaces \cite[pp.
2-4]{Verona}, \cite[\S 1.5]{GorMacPh}, \cite[\S 1.4.1]{Pflaum}.

\begin{definition}
\label{defn:LocalConeBundles}
A subspace $S$ of a stratified space $Y$ has
{\em smooth, local cone bundle neighborhoods\/} if
the following holds.
For each stratum $S_i$, there are
\begin{itemize}
\item
A neighborhood $\sO_i$ of $S_i$ in $Y$ with $\sO_i\cap\sO_j$
non-empty if and only if $i<j$ or $j<i$,
\item
A fiber bundle $\pi_i: N_i\to S_i$ with fiber $F_i$, a cone,
wherein we identify $S_i$ with the section of $\pi_i$ given
by the cone point,
\item
A homeomorphism $\bga_i$ from a neighborhood of $S_i$ in $N_i$
with $\sO_i$.
\end{itemize}
Let $t_i:\sO_i\to [0,\infty)$ be the function defined by the
composition of $\bga_i^{-1}$ and the cone parameter on $N_i$.
We will also write $\pi_i:\sO_i\to S_i$ for the map $\pi_i\circ\bga_i^{-1}$.
These maps satisfy the {\em Thom-Mather control conditions\/} if
\begin{itemize}
\item
The map $(\pi_i,t_i):\sO_i\to S_i\times [0,\infty)$ is a smooth
submersion on each stratum,
\item
For $i<j$, on $\sO_i\cap\sO_j$,
\begin{equation}
\label{eq:TMControl}
\pi_i\circ\pi_j=\pi_i,\quad
t_i\circ\pi_j=t_i.
\end{equation}
\end{itemize}
We say the control data $\{(\sO_i,\pi_i,t_i)\}$ has
{\em compatible structure groups\/} if
\begin{itemize}
\item
The structure group of each bundle $\pi_i:N_i\to S_i$
is a compact Lie group $G_i$,
\item
For $i<j$, the intersection $\sO_i\cap\sO_j$
is a $G_i$-subbundle of $N_i$ and a $G_j$-subbundle
of $N_j$,
\item
For $i<j$, on the intersection $\sO_i\cap\sO_j$, the level sets
$t_j^{-1}(\eps)$ are $G_i$-subbundles.
\end{itemize}
\end{definition}
For spaces satisfying Definition \ref{defn:LocalConeBundles},
we define the link of $S$ in $Y$ by,
$$
\bL=\rd(\cup_i \sO_i).
$$
If the intersection of $\bL$ with the lower strata has codimension
greater than or equal to two in $\bL$, then $\bL$ has a fundamental
class $[\bL]$.
The intersection number of a geometric representative or divisor $\sV$ with
$\bL$ can be represented, through a duality argument, as a
cohomological pairing:
$$
\#(\sV\cap\bL)=\langle \mu,[\bL]\rangle,
$$
with an appropriate cohomology class $\mu$. To give a partial
computation of this pairing, even when we do not know the topology
of the fibers $F_i$, we decompose $\bL$ as
\begin{equation}
\label{eq:DecomposeL} \bL=\cup_i\ \bL_i, \quad\text{where}\quad
\bL_i=\bga_i\left(t_i^{-1}(\eps_i)\right) -\cup_{j\neq
i}\bga_j\left(t_j^{-1}[0,\eps_j)\right).
\end{equation}
For generic choices of the constants $\eps_i$, the components $\bL_i$
of this decomposition
will be smoothly-stratified, closed, codimension-zero subspaces of
$\bL$ in which each stratum is a smooth manifold with corners (see
\cite[p. 7]{KirbySiebenmann} or \cite[Definition 1.2.2]{RoigDom}).
The boundary of each component $\bL_i$ can be described as:
$$
\rd\bL_i=
\cup_{j\neq i}\ \bL_i\cap\bL_j
=
\cup_{j\neq i}\ \bL_i\cap t_j^{-1}(\eps_j).
$$
Because of the control on the overlaps $\sO_i\cap\sO_j$ given by the
assumption on compatible structure groups, each component $\bL_i$
is a $G_i$ subbundle of $N_i\to S_i$ with fiber which we will write as
$F_i(\beps)$, appearing in the diagram
\begin{equation}
\label{eq:PushPull}
\begin{CD}
\bL_i @> \tilde f_i>> \EG_i\times_{G_i} F_i(\beps)
\\
@V\pi_i  VV @V p_i VV
\\
S_i @>f_i >> \BG_i,
\end{CD}
\end{equation}
Let $\mu_i$ be the restriction of $\mu$ to $\bL_i$.
Assume that $\mu_i$ is given by a product of classes pulled back from $S_i$
and of $G_i$-equivariant cohomology classes on $F_i(\beps)$.
Then, the assumption on compatible structure groups allows us
to chose a representative of $\mu$ such that the restrictions
$\mu_i$ have compact support on $\bL_i$ and are  given
by a product, $\mu_i=\pi_i^*x_i\smile\tilde f_i^*\nu_i$
where $x_i\in H^\bullet_c(S_i)$ and
$$
\nu_i\in H^\bullet(\EG_i\times_{G_i}F(\beps)))
$$
has compact vertical support with respect to $p_i$.
Then the decomposition of $\bL$ in
\eqref{eq:DecomposeL} yields the equalities
\begin{equation}
\label{eq:Decomposition}
\begin{aligned}
\langle \mu,[\bL]\rangle
{}&=
\sum_i \langle \mu,[\bL_i]\rangle\\
{}&=
\sum_i \langle( \pi_i)_*\mu_i, [S_i]\rangle \\
{}&=
\sum_i \langle( \pi_i)_*(\pi_i^*x_i\smile \tilde f_i^* \nu_i ), [S_i]\rangle
\\
{}&=
\sum_i\langle x_i\smile f_i^*((p_i)_*\nu_i),[S_i]\rangle.
\end{aligned}
\end{equation}
The final step in \eqref{eq:Decomposition} is known as a
pushforward-pullback argument (see \cite[Proposition
1.15]{TauFloer}).

This argument allows us to isolate the topology of the fibers $F_i$
in universal constants, producing the desired partial computation of
the intersection number $\#(\sV\cap\bL)$ in terms of the homotopy
class of the classifying map $f_i:S_i\to \BG_i$ without explicit
knowledge of the fibers $F_i$.

\subsection{The proof of Equation \eqref{eq:ReducibleLink} and of the
Kotschick-Morgan conjecture}
\label{subsec:ProofOfRedLink}

The proofs of Equation \eqref{eq:ReducibleLink} and the
Kotschick-Morgan Conjecture \cite{KotschickMorgan} both require a
partial computation of an intersection number with the link of a
closed subspace of a stratified space of the type described in \S
\ref{subsec:Model}. For Equation \eqref{eq:ReducibleLink}, the
stratified space is $\sM_{\ft}/S^1$ and the subspace is
$M_{\fs}\times\Sym^\ell(X)$. For the Kotschick-Morgan conjecture,
the stratified space is the Uhlenbeck compactification of a
parameterized moduli space of anti-self-dual connections, $\bar
M^w_\ka(g_I)$, parameterized by a path of metrics $g_I$, while the
subspace is $[A_0]\times\Sym^\ell(X)$ where $[A_0]$ is a reducible,
anti-self-dual connection.

The strata of the closed subspaces, $M_{\fs}\times\Sym^\ell(X)$ and
$[A_0]\times\Sym^\ell(X)$, are given by $M_{\fs}\times\Si$ and
$[A_0]\times\Si$, respectively, where $\Si\subset\Sym^\ell(X)$ is a
smooth stratum.

The gluing theorems of \cite{FL3} do not quite yield the cone bundle
neighborhoods of $M_{\fs}\times\Si$ in $\bar\sM_{\ft}/S^1$ required
in Definition \ref{defn:LocalConeBundles}. Rather they provide a
{\em virtual cone bundle neighborhood\/} by which we mean:
\begin{itemize}
\item
A cone bundle $\Gl(\ft,\fs,\Si)\to M_{\fs}\times\Si$,
\item
An obstruction section, $\fo_\Si$, of a pseudo-vector bundle
$\Upsilon_{\ft,\fs}\to\Gl(\ft,\fs,\Si)$,
\item
A homeomorphism between
$\fo_\Si^{-1}(0)$ and a neighborhood of $M_{\fs}\times\Si$ in
$\bar\sM_{\ft}/S^1$.
\end{itemize}
In \cite{FL5}, we define $\bL^{\vir}_{\ft,\fs}$ as the boundary
of the union of cone bundles.
The actual link $\bL_{\ft,\fs}$ is the intersection of $\bL^{\vir}_{\ft,\fs}$
with the zero locus of the obstruction sections.
The intersection number
in \eqref{eq:ReducibleLink} is then equal to
$$
\langle \mu\smile e(\Upsilon_{\ft,\fs}),[\bL^{\vir}_{\ft,\fs}]\rangle,
$$
where $e(\Upsilon_{\ft,\fs})$ acts as an Euler class of the
pseudo-vector bundle and $\mu$ is a cohomology class dual to the
geometric representatives appearing in \eqref{eq:ReducibleLink}. In
\cite{FL5} we show that the virtual cone bundle neighborhoods can be
deformed, in a sense to be defined in \S
\ref{subsec:DeformingTheFiber}, to satisfy the Thom-Mather control
condition and the compatible structure group condition of Definition
\ref{defn:LocalConeBundles}. Then, the arguments of \S
\ref{subsec:Model} can be applied to compute the cohomology pairing
above and thus the intersection number in Equation
\eqref{eq:ReducibleLink}.

The methods of \cite{FL5} also apply to the Kotschick-Morgan conjecture.
The gluing maps of Taubes, \cite{TauIndef,FLKM1}, give the smooth,
local cone
bundle neighborhoods of $[A_0]\times \Si$ in $\bar M^w_\ka(g_I)$.
The constructions of \cite{FL5}, while
couched in the language of $\SO(3)$ monopoles, easily
translate to the language of anti-self-dual connections and
show that these cone bundle neighborhoods can also be deformed so that
they satisfy the Thom-Mather control condition and
the compatible structure group condition of Definition
\ref{defn:LocalConeBundles}.  The pushforward-pullback
argument described in \S \ref{subsec:Model} would then yield a
proof of the Kotschick-Morgan conjecture.

While the existence of cone bundle neighborhoods for
$[A_0]\times\Si$ and the virtual cone bundle neighborhoods for
$M_{\fs}\times\Si$ has been known since \cite{TauIndef,FLKM1,FL3},
the result in \cite{FL5} that these neighborhoods can be deformed to
satisfy the Thom-Mather control condition and the compatible
structure group condition is new. The authors believe that any
attempt to prove Equation \eqref{eq:ReducibleLink} or the
Kotschick-Morgan conjecture must solve these issues.

As an example, in \cite{BohuiChen}, Chen attempts to prove the
Kotschick-Morgan conjecture by constructing a bubbletree resolution
of the Uhlenbeck compactification $\bar M^w_\ka(g_I)$ and applying
equivariant localization arguments to this resolution. To prove that
this resolution is a $C^1$ orbifold, Chen tries to show that, for
gluing maps, $\bga_\Si$ and $\bga_{\Si'}$, parameterizing
neighborhoods of different strata, the transition map
$\bga_{\Si}^{-1}\circ\bga_{\Si'}$ is smooth. At first glance, this
might appear to be a weaker result than the above requirements on
Thom-Mather control conditions and compatible structure groups.
However, Chen compares the gluing maps $\bga_\Si$ and $\bga_{\Si'}$
by introducing an artificial transition map, $\Phi_{\Si,\Si'}$,
between the domains of the two gluing maps and defining an
isotopy between $\bga_\Si$ and $\bga_{\Si'}\circ\Phi_{\Si,\Si'}$.
This artificial transition map, similar to that introduced in
\cite{KotschickMorgan}, has the properties necessary to show
Definition \ref{defn:LocalConeBundles} holds.
Thus, a complete
implementation of the method of \cite{BohuiChen} would yield
essentially the same program as that of \cite{FL5}, with the
additional complication of having to work with the extra
data of the bubbletree compactification.

The method of comparing a gluing map $\bga_\Si$ with a composition
$\bga_{\Si'}\circ\Phi_{\Si,\Si'}$ by constructing an isotopy between
the two also appears in \cite[\S 4.5.1]{LenessWC}. This is a very
natural method to try because a direct comparison of the gluing maps
$\bga_\Si$ and $\bga_{\Si'}$ appears impractical for reasons
discussed in the beginning of \S \ref{sec:Overlaps}. The comparison
becomes significantly more difficult when there are more than two
open sets, for rather than constructing a single isotopy between two
maps one must now construct a family of diffeomorphisms
parameterized by a higher-dimensional simplex with the maps
$\bga_\Si$ and $\bga_{\Si_i}\circ\Phi_{\Si,\Si_i}$ at the vertices,
in a manner similar to the development of a $k$-isotopy \cite[p.
182]{Hirsch}. The authors believe that the inductive constructions
described in \S \ref{subsec:DeformingTheFiber} could be used to give
such a $k$-isotopy.

\subsection{From Theorem \ref{conj:PTConj} to the Witten Conjecture and other applications}
Equation \eqref{eq:MainThm} shows that the Donaldson invariants are
determined by the Seiberg-Witten invariants and the homotopy type of
$X$, but the relation does not immediately yield Witten's formula.
We describe how Witten's relation between the Donaldson and
Seiberg-Witten invariants follows from Theorem \ref{conj:PTConj} in
a separate article \cite{FLWConjecture}.

As noted in the preceding section, the proof of Equation
\eqref{eq:ReducibleLink} can be adapted to give a proof of the
Kotschick-Morgan conjecture.

It is also worth noting that Kronheimer and Mrowka's proof
\cite{KMPropertyP} of Property P relies on Equation
\eqref{eq:ReducibleLink}.

\subsection{Organization}
This article comprises the following sections. In \S
\ref{sec:Prelim}, we review the basic definitions and ideas of the
$\SO(3)$-monopole program and describe the intersection numbers in
Equation \eqref{eq:ReducibleLink}. In \S \ref{sec:GluingMaps}, we
describe the gluing maps we use to parameterize the neighborhoods
containing these intersection numbers. In \S \ref{sec:Overlaps}, we
summarize the proofs in \cite{FL5} of the control properties of the
gluing map overlaps. Finally, in \S
\ref{sec:CohomolFormalism}, we use this understanding of the
overlaps of the gluing maps to introduce a cohomological formalism
to compute the desired intersection numbers.

{\em Acknowledgements:} This note is an expansion of the second
author's talk at the Fields Institute Conference on Geometry and
Topology of Manifolds in May of 2004.  We are grateful to the
conference organizers for inviting us.

\section{Preliminaries}
\label{sec:Prelim} Throughout this article, let $X$ be a smooth,
closed, and oriented four-manifold.  We will assume that $b_1(X)=0$
and $b^+(X)\ge 2$ for the sake of simplicity, although much of what
we discuss here holds for more general $b^1(X)$ and for $b^+(X)=1$.
The material reviewed in this section appears in full detail in
\cite{FL2a,FL2b,FeehanGenericMetric}.

\subsection{\Spinu structures}
\label{subsec:SpinuStr}
A Clifford module  for $T^*X$ is defined by a complex vector
bundle $V\to X$ and a Clifford multiplication map, $\rho:T^*X\to
\End_\CC(V)$ which is a real-linear bundle map satisfying
\begin{equation}
\label{eq:CliffordMapDefn} \rho(\alpha)^2 =
-g(\alpha,\alpha)\id_{V} \quad\text{and}\quad \rho(\alpha)^\dagger
= -\rho(\alpha), \quad \alpha \in C^\8(T^*X).
\end{equation}
The map $\rho$ uniquely extends to a linear isomorphism,
$\rho:\Lambda^{\bullet}(T^*X)\otimes_\RR\CC\to\End_\CC(V)$, and
gives each fiber $V_x$ the structure of a Hermitian Clifford module
for the complex Clifford algebra $\CCl(T^*X)|_x$, for all $x\in X$.
There is a splitting $V=V^+\oplus V^-$, where $V^\pm$ are the $\mp
1$ eigenspaces of $\rho(\vol)$.

A {\em \spinc structure\/} is then a Clifford module for $T^*X$,
$\fs=(\rho,W)$, where $W$ has complex rank four; it defines a
class
\begin{equation}
\label{eq:DefineChernClassOfSpinc}
c_1(\fs) = c_1(W^+) \in
H^2(X;\ZZ).
\end{equation}
We call a Clifford module for $T^*X$, consisting of a pair $\ft
\equiv (\rho,V)$,  a {\em \spinu structure\/} when $V$ has complex
rank eight.  If $\fs=(\rho_W,W)$ is a \spinc structure, then for any
\spinu structure, $\ft=(\rho,V)$, there is a complex, rank-two
vector bundle $E\to X$ with $(\rho,V)=(\rho_W\otimes\id_E,W\otimes
E)$ \cite[Lemma 2.3]{FL2a}.

A \spinu structure, $\ft=(\rho,V)$, defines some auxiliary bundles
over $X$. Recall that $\fg_{V}\subset\su(V)$ is the $\SO(3)$
subbundle given by the span of the sections of the bundle $\su(V)$
which commute with the action of $\CCl(T^*X)$ on $V$. The fibers
$V_x^+$ define complex lines whose tensor-product square is
$\det(V^+_x)$ and thus a complex line bundle over $X$,
\begin{equation}
\label{eq:CliffordDeterminantBundle} {\det}^{\frac{1}{2}}(V^+).
\end{equation}
A \spinu structure $\ft$ thus defines classes,
\begin{equation}
\label{eq:SpinUCharacteristics} c_1(\ft)=\textstyle{\frac{1}{2}}
c_1(V^+), \quad p_1(\ft) = p_1(\fg_{V}), \quad \text{and}\quad
w_2(\ft)=w_2(\fg_{V}).
\end{equation}
If $\fs=(\rho,W)$ is a \spinc structure and $V\cong W\otimes_\CC
E$, then
$$
\fg_V = \su(E) \quad\text{and}\quad {\det}^{\frac{1}{2}}(V^+) =
\det(W^+)\otimes_\CC\det(E).
$$

\subsection{The equations}
A unitary connection $A$ on $V$ is {\em spin \/} if
\begin{equation}
\label{eq:SpinConnection} [\nabla_A,\rho(\alpha)]
=\rho(\nabla\alpha) \quad\text{on }C^\8(V),
\end{equation}
for any $\alpha\in C^\8(T^*X)$, where $\nabla$ is the Levi-Civita
connection. We will write $\sA_{\ft}$ for the space of spin
connections on $V$ which induce a fixed connection on $\det(V)$.
There is a bijection between $\sA_\ft$ and the space of orthogonal
connections on $\fg_V$.

We will write $\sG_{\ft}$ for the space of gauge transformations of
$V$ which commute with Clifford multiplication and which have
Clifford-determinant equal to one \cite[Definition 2.6]{FL2a}. If
$V=W\otimes E$ as above, then $\sG_{\ft}$ is the set of gauge
transformations of $V$ induced by the special-unitary gauge
transformations of $E$.

We will consider pairs in the Sobolev completions of the space
$$
\tsC_{\ft}=\sA_{\ft}\times\Om^0(V^+),
$$
and points in the quotient,
$$
\sC_{\ft}=\tsC_{\ft}/\sG_{\ft}.
$$
For $(A,\Phi)\in\tsC_{\ft}$, the $\SO(3)$-monopole equations are
then:
\begin{equation}
\begin{aligned}
\rho(F_A^+)_0&=(\Phi\otimes\Phi^*)_{00},
\\
D_A\Phi&=0.
\end{aligned}
\label{eq:SO3}
\end{equation}
Here, $D_A:\Om^0(V^+)\to\Om^0(V^-)$ is the Dirac operator associated
to the connection $A$, while $(\Phi\otimes\Phi^*)_{00}$ denotes the
doubly trace-free component of $\Phi\otimes\Phi^*$ (see \cite{FL2a}
for details). The moduli space of $\SO(3)$ monopoles on $\ft$ is
$$
\sM_{\ft} = \{[A,\Phi]\in\sC_{\ft}: \text{\eqref{eq:SO3} holds}\}.
$$
We write $\sM^0_{\ft}$ for the subspace of $\sM_{\ft}$ where the
section $\Phi$ is not identically zero. We let $\sM^*_{\ft}$ denote
the set of pairs $[A,\Phi]\in\sM_{\ft}$ where $A$ does not admit a
reduction as $A=A_1\oplus A_2$ with respect to a splitting
$V=(W\otimes L_1)\oplus (W\otimes L_2)$ for line bundles $L_1$ and
$L_2$. Then, for generic perturbations \cite{FeehanGenericMetric,
TelemanGenericMetric}, the space
$$
\sM^{*,0}_{\ft}=\sM^*_{\ft}\cap\sM^0_{\ft}
$$
is a smooth manifold.

\subsection{The singularities}
The $S^1$ action,
\begin{equation}
\label{eq:S1Action}
\left( e^{i\theta},[A,\Phi]\right) \mapsto
[A,e^{i\theta}\Phi]
\end{equation}
has stabilizer $\{\pm 1\}$ on $\sM^{*,0}_{\ft}$.
There are two
types of fixed points for this action. To describe these, it helps
to assume that $V=W\otimes E$ where $\fs=(\rho,W)$ is a \spinc
structure.

The first type of fixed point occurs when the section $\Phi$ is
identically zero. The subspace of such pairs is diffeomorphic to
the moduli space of anti-self-dual connections on the $\SO(3)$
bundle $\fg_V=\su(E)$:
$$
\sM_{\ft}-\sM^0_{\ft}=M^w_\ka,
$$
where $M^w_\ka$ is the moduli space of anti-self-dual connections on
$\fg_V$, $w=c_1(E)$, and $\ka=p_1(\ft)$ \cite{FL2a}.

The second type of fixed point occurs when the connection $A$ is
reducible in the sense that it can be written as $A=A_1\oplus A_2$
with respect to a splitting
$$
V\simeq W\otimes E\simeq  (W\otimes L_1)\oplus (W\otimes L_2).
$$
The subspace of such fixed points is diffeomorphic to a
perturbation of the Seiberg-Witten moduli space associated to the
\spinc structure $\fs\otimes L_i=(\rho,W\otimes L_i)$ (either
$i=1$ or $i=2$) as discussed in \cite{FL2a}.  These \spinc
structures lie in the set
$$
\Red(\ft)= \{\fs\in\Spinc(X): \left(
c_1(\fs)-c_1(\ft)\right)^2=p_1(\ft)\}.
$$
If $M_{\fs}$ is the Seiberg-Witten moduli space associated to the
\spinc structure $\fs$, then
$$
\sM_{\ft}-\sM^*_{\ft}=\cup_{\fs\in\Red(\ft)} M_{\fs}.
$$
The moduli space $\sM_{\ft}/S^1$ then gives a cobordism between the
links, $\bL^{\asd}_{\ft}$ of $M^w_\ka$ in $\sM_{\ft}$ and $
\bL_{\ft,\fs}$ of $M_{\fs}$ in $\sM_{\ft}$. The dimension of this
cobordism is,
\begin{equation}
\label{eq:ModuliDim1}
\dim \sM^{*,0}_{\ft}/S^1
=\dim M^w_\ka +2n_a(\ft)-2,
\end{equation}
where
\begin{equation}
\label{eq:ModuliDim}
n_a(\ft)=\ind_\CC D_A=\frac{1}{4}\left(p_1(\ft)+c_1(\ft)^2-\si(X)\right).
\end{equation}
Unfortunately, the cobordism  $\sM_{\ft}/S^1$ is not compact.

\subsection{The compactification}
The moduli space $\sM_{\ft}$ admits a compactification similar to
the Uhlenbeck compactification of the moduli space of anti-self-dual
connections. Let $\ft(\ell)$ be the \spinu structure with
$c_1(\ft(\ell))=c_1(\ft)$ and $p_1(\ft(\ell))=p_1(\ft)+4\ell$. Thus,
if we can write $V=W\otimes E$, where $\fs=(\rho,W)$ is a \spinc
structure, then $\ft(\ell)=(\rho,W\otimes E_\ell)$ where
$c_1(E_\ell)=c_1(E)$ and $c_2(E_\ell)=c_2(E)-\ell$. Let
$\Sym^\ell(X)$ denote the $\ell$-th symmetric product of $X$,
$\Sym^\ell(X)=X^\ell/\fS_\ell$. Define
$$
I\sM_{\ft}=\sM_{\ft}\cup \left(\cup_{\ell=1}^N
\sM_{\ft(\ell)}\times\Sym^\ell(X)\right),
$$
and give this space the topology induced by Uhlenbeck convergence of
sequences. That is, a sequence $\{[A_\alpha,\Phi_\alpha]\}$
converges to a point $([A_0,\Phi_0],\bx)$ where $\bx\in\Sym^\ell(X)$
if, after gauge transformations, the restrictions of
$(A_\alpha,\Phi_\alpha)$ to compact subsets of $X-\bx$ converge to
the same restriction of $(A_0,\Phi_0)$ in the smooth topology. In
addition, the sequence of measures defined by $|F_{A_\alpha}|^2$
must converge, in the weak star topology, to that defined by
$|F_{A_0}|^2$ added to a multiple of the Dirac delta measure
supported at $\bx$.

Define $\bar\sM_{\ft}$ to be the closure of $\sM_{\ft}$ in
$I\sM_{\ft}$ with the topology described above.

The dimension formula in Equation \eqref{eq:ModuliDim1} implies that
\begin{equation}
\label{eq:LowerStrataDimen} \dim\left(
\sM^{*,0}_{\ft(\ell)}\times\Sym^\ell(X)\right) =
\dim\sM^{*,0}_{\ft}-2\ell,
\end{equation}
which will be useful in dimension-counting arguments.

We will also use the space
\begin{equation}
\label{eq:ConfigComplete}
\bar\sC_{\ft} = \sC_{\ft}\cup
\left(\cup_{\ell=1}^N \sC_{\ft(\ell)}\times\Sym^\ell(X)\right),
\end{equation}
with the same definition of convergence.

The $S^1$ action \eqref{eq:S1Action} extends continuously to the
compactification $\bar\sM_{\ft}$ with similar fixed point sets. That
is, the fixed point subspaces again divide into two types, those
where the section vanishes and those where the connection is
reducible.  The subspace of $\bar\sM_{\ft}$ where the section
vanishes is given by $\bar M^w_\ka$, the Uhlenbeck compactification
of the moduli space of anti-self-dual connections on $\fg_V$.

The fixed point subspaces where the connection is reducible are more
complicated when they lie in the compactification. A level-$\ell$
reducible is of the form
$$
M_{\fs}\times\Sym^\ell(X) \subset
\sM_{\ft(\ell)}\times\Sym^\ell(X),
$$
where $M_{\fs}\subset\sM_{\ft(\ell)}$ and thus
$(c_1(\fs)-c_1(\ft))^2=p_1(t)+4\ell$.  We then define
\begin{align*}
\Red_\ell(\ft) {}&= \{\fs\in\Spinc(X):
(c_1(\fs)-c_1(\ft))^2=p_1(\ft)+4\ell\},
\\
\bar{\Red}(\ft) {}&= \cup_{\ell\ge 0} \Red_\ell(\ft).
\end{align*}
The space $\bar\sM_{\ft}/S^1$ defines a smoothly-stratified
cobordism between the link of $\bar M^w_\ka$ in $\bar\sM_{\ft}/S^1$
and the links of $M_{\fs}\times\Sym^\ell(X)$ in $\bar\sM_{\ft}/S^1$,
where $\fs\in\bar{\Red}(\ft)$.

\subsection{The cobordism and cohomology classes}
Let $\bar \bL^{\asd}_{\ft}$ be the link of the Uhlenbeck
compactification of the moduli space of anti-self-dual
connections, $\bar M^w_\ka$, in $\bar\sM_{\ft}/S^1$. For
$\fs\in\bar{\Red}_\ell(\ft)$, let $\bL_{\ft,\fs}$ be the link
of $M_{\fs}\times\Sym^\ell(X)$ in $\bar\sM_{\ft}/S^1$.

Following \cite{KMStructure} and recalling that we assumed
$b_1(X)=0$, define
$$
\AAA(X)=\Sym^\bullet(H_2(X);\RR)\oplus H_0(X;\RR)),
$$
and  assign to $\beta\in H_\bullet(X)$ the degree
$\deg(\beta)=4-\bullet$ in $\AAA(X)$. For a monomial $z$ in
$\AAA(X)$, a geometric representative $\bar\sV(z)$ was constructed
in \cite{KMStructure} whose intersection number with $\bar M^w_\ka$
defined the Donaldson invariant. In \cite{FL2a,FL2b}, geometric
representatives extending $\bar\sV(z)$ from $\bar M^w_\ka$ to
$\bar\sM^*_{\ft}/S^1$ were defined. In addition, a geometric
representative $\bar\sW$ was defined on $\bar\sM^{*,0}_{\ft}/S^1$
which is dual to a multiple of the first Chern class of the $S^1$
action.

For  $\deg(z)=\dim M^w_\ka$, we computed the following intersection
number in \cite{FL2b},
$$
2^{n-1} D^w_X(z)=\#\left( \bar\sV(z)\cap\bar\sW^{n-1}\cap
\bar\bL^{\asd}_{\ft}\right),
$$
where $n=n_a(\ft)$. The geometric representatives intersect the
lower strata of $\bar\sM_{\ft}/S^1$ transversely away from the fixed
point sets of the $S^1$ action.  Hence, the dimension formula
\eqref{eq:LowerStrataDimen} shows that the cobordism provided by
$\bar\sM_{\ft}/S^1$ and the preceding identity yield the following
expression for the Donaldson invariant:
\begin{equation}
\label{eq:CobordismSum1}
2^{n-1} D^w_X(z) =
-\sum_{\fs\in\Spinc(X)} \#\left(
\bar\sV(z)\cap\bar\sW^{n-1}\cap\bL_{\ft,\fs} \right),
\end{equation}
where we define the link $\bL_{\ft,\fs}$ to be empty if
$\fs\notin\bar{\Red}(\ft)$.
The computation of the intersection
number
\begin{equation}
\label{eq:ReducibleIntersectionNumber} \#\left(
\bar\sV(z)\cap\bar\sW^{n-1}\cap\bL_{\ft,\fs} \right)
\end{equation}
appears in \cite{FL2b} for $\fs\in\Red(\ft)$ and appears in
\cite{FLLevelOne} for $\fs\in\Red_1(\ft)$. For
$\fs\in\Red_\ell(\ft)$ with $\ell\ge 2$, the computation becomes
significantly more difficult, although the development in
\cite{LenessWC} suggests an approach to the case $\ell=2$. In
\cite{FL5}, we show how the following version of Equation
\eqref{eq:ReducibleLink} follows from the technical result
Theorem  \ref{thm:ExtendedGluingThm}.

\begin{thm}
\label{conj:PTConj} Assume the result of Theorem
\ref{thm:ExtendedGluingThm}. If $b_1(X)=0$, if $\ft$ is a \spinu
structure on $X$ and $\fs\in \Spinc(X)$ with
$$
M_{\fs}\times\Sym^\ell(X) \subset I\sM_{\ft},
$$
then
\begin{align*}
{}&\# \left( \bar\sV(h^{\delta-2m}x^m)\cap \bar\sW^{n-1} \cap
\bL_{\ft,\fs}\right)
\\
\qquad &= \SW_X(\fs) \sum_{i=0}^k p_{\delta,\ell,m,i}(A,B)Q_X^i(h),
\end{align*}
where $k=\min[\ell,(\delta-2m)/2)]$,
\begin{align*}
A&=\langle c_1(\fs)-c_1(\ft),h\rangle,
\\
B&=\langle c_1(\ft),h\rangle,
\end{align*}
and $p_{\delta,\ell,m,i}$ is a homogeneous polynomial of degree
$(\delta-2m-2i)$ whose coefficients are universal functions of
$$
\chi(X),\ \si(X),\ c_1(\fs)^2,\
c_1(\ft)^2,\ c_1(\ft)\cdot c_1(\fs),\ p_1(\ft),\
m,\ \delta,\ \text{and}\ \ell.
$$
\end{thm}

\begin{rmk}
A similar result holds when $b^1(X)>0$, but we omit it for
simplicity.
\end{rmk}

\section{The gluing maps}
\label{sec:GluingMaps}
To prove Theorem \ref{conj:PTConj}, we show
that for each smooth stratum $\Si$ of $\Sym^\ell(X)$, the subspaces
$M_{\fs}\times\Si$ of $\bar\sM_{\ft}/S^1$ admit  virtual local cone
bundle neighborhoods in the sense of \S \ref{subsec:ProofOfRedLink}
which satisfy Definition \ref{defn:LocalConeBundles}. The gluing
theorems of \cite{FL3} provide these neighborhoods by defining a
``gluing map'' which parameterizes a neighborhood of
$M_{\fs}\times\Si$ in $\bar\sM_{\ft}$.

A gluing map is a composition of a splicing map and a gluing
perturbation which we describe in the following sections.

\subsection{Splicing maps}
\label{subsec:SplicingMap}
The splicing map creates an approximate
solution to the $\SO(3)$-monopole equations by attaching solutions
in $\sM_{\ft(\ell)}$ to concentrated solutions on $S^4$ in a kind of
connected sum construction. The domain of the splicing map contains
the following data:
\begin{enumerate}
\item A `background pair' $(A_0,\Phi_0)\in\sM_{\ft(\ell)}$, \item
Splicing points,  $\bx\in \Si$, \item Solutions of the equations
on $S^4$.
\end{enumerate}
We now describe the spaces in which these data live.

\subsubsection{The background pair}
In \cite{FL2a}, a neighborhood of $M_{\fs}$ in $\sM_{\ft(\ell)}$
was described using virtual neighborhood techniques.  That is, we
defined
\begin{enumerate}
\item A pair of vector bundles, $N_{\ft(\ell),\fs}\to M_{\fs}$ and
$\Upsilon_{\fs}\to M_{\fs}$, \item An $S^1$-equivariant embedding
$\ga_{\fs}:N_{\ft(\ell),\fs}\to \sC_{\ft(\ell)}$ which is the
identity on $M_{\fs}$, \item A section $\fo_{\fs}$ of the pullback
of $\Upsilon_{\fs}$ to $N_{\ft(\ell),\fs}$
\end{enumerate}
such that the restriction of $\bga_{\fs}$ to $\fo_{\fs}^{-1}(0)$
defines an $S^1$-equivariant homeomorphism from
$\fo_{\fs}^{-1}(0)$ onto a neighborhood of $M_{\fs}$ in
$\sM_{\ft(\ell)}$.

\subsubsection{The splicing points}
Let $\Si\subset\Sym^\ell(X)$ be a stratum given by the partition
$\ka_1+\dots+\ka_r=\ell$. A point $\bx\in\Si$ can be described by
an unordered collection of distinct points
$\{x_1,\dots,x_r\}\subset X$ with the multiplicity $\ka_i$
attached to $x_i$. We will write $\Si<\Si'$ to indicate that
$\Si\subset \cl(\Si')$.

\subsubsection{The $S^4$ connections}
\label{subsubsection:S4Conn}
Solutions of the $\SO(3)$-monopole
equations on $S^4$ correspond to anti-self-dual connections on an
$\SU(2)$ bundle $E_\ka\to S^4$. In addition, the splicing process
requires a frame for $E_\ka|_s$ where $s\in S^4$ is the south pole.
We write $M^s_\ka(S^4)$ for the moduli space of framed,
anti-self-dual connections on $E_\ka$.

We also require these connections to be mass-centered, in the
sense that
$$
\int_{S^4} \vec x |F_A|^2 d\vol=0,
$$
where $\vec x:S^4-\{s\}\to\RR^4$ is given by stereographic
projection.

Finally, we require these connections to be concentrated near the
north pole in the sense that
\begin{equation}
\label{eq:DilationParameter}
\int_{S^4} |\vec x |^2|F_A|^2 d\vol \le
\eps,
\end{equation}
for some small constant $\eps$ which shall not be specified here.

The connections on $S^4$ then lie in the space,
$$
\bar M^{s,\diamond}_\ka(S^4),
$$
which is defined to be the Uhlenbeck compactification of the moduli
space of framed, mass-centered, sufficiently concentrated,
anti-self-dual connections on $E_\ka\to S^4$. The space $\bar
M^{s,\diamond}_\ka(S^4)$ is a cone with cone parameter squared given
by the left-hand side of the inequality
\eqref{eq:DilationParameter}. In addition, one can show that $\bar
M^{s,\diamond}_\ka(S^4)$ is a sub-analytic and thus a Thom-Mather
stratified space.

\subsubsection{Describing the map}
\label{subsubsec:SplicingMap} Let $\Si\subset\Sym^\ell(X)$ be the
stratum given by the partition $\ell=\ka_1+\dots+\ka_r$. Let $\bA$
denote the vector of data prescribed by
\begin{enumerate}
\item $(A_0,\Phi_0)\in \bga_{\fs}(N_{\ft(\ell),\fs})$, \item
$\bx\in\Si$, with $\bx=\{x_1,\dots,x_r\}\subset X$, and \item For
$i=1,\dots,r$, $[A_i,F^s_i]\in \barM^{s,\diamond}_{\ka_i}(S^4)$.
\end{enumerate}
We define the image of the splicing map to be
$$
(A',\Phi') = \bga_\Si'(\bA) = (A_0,\Phi_0)\#_{x_1}(A_1,0)\#\dots
\#_{x_r}(A_r,0) ,
$$
where, roughly speaking,
$$
(A',\Phi')(x) \approx
\begin{cases}
(A_0,\Phi_0)(x) & \text{for $x$ away from $x_i$},\\
(A_i,0)(x) & \text{for $x$ near $x_i$}.
\end{cases}
$$
The precise definition of the splicing map requires careful use of
trivializations of the relevant bundles and cut-off functions to
interpolate between the given pairs of sections and connections.
Nonetheless, it is a completely explicit map.

\subsubsection{The domain of the map}
\label{subsubsec:DomainOfSplicingMap}
If the stratum $\Si$ is given by the partition
$\ka_1+\dots+\ka_r=\ell$, then the domain, $\Gl(\ft,\fs,\Si)$, of
the splicing map $\bga_\Si'$ described in \S
\ref{subsubsec:SplicingMap} can be described as a fiber bundle:
\begin{equation}
\label{eq:LocalFiberBundle}
\begin{CD}
\prod_{i=1}^r \bar M^{s,\diamond}_{\ka_i}(S^4) @>>> \Gl(\ft,\fs,\Si)
\\
@. @V \Pi_\Si VV
\\
@. N_{\ft(\ell),\fs}\times\Si
\end{CD}
\end{equation}
with underlying principle bundle determined by the homotopy type
of $X$ and the characteristic classes $c_1(\fs)$, $c_1(\ft)$, $p_1(\ft)$.
Note that because the spaces $\bar M^{s,\diamond}_{\ka_i}(S^4)$ are cones,
the fiber bundle \eqref{eq:LocalFiberBundle} is a cone bundle in the sense
of Definition \ref{defn:LocalConeBundles}.
We further note that \eqref{eq:LocalFiberBundle} is a
non-trivial bundle because of the need to chose trivializations of
the relevant bundles in the splicing map.
We will write $\pi_\Si$ for the composition
\begin{equation}
\label{eq:LocalFiberProjection}
\pi_\Si:\Gl(\ft,\fs,\Si)
\to
N_{\ft(\ell),\fs}\times\Si
\to
M_{\fs}\times\Si
\end{equation}
of the projection in \eqref{eq:LocalFiberBundle} with the obvious projection
defined by the vector bundle $N_{\ft(\ell),\fs}\to M_{\fs}$.

\subsection{Gluing perturbations}
The pair $\bga_\Si'(\bA)$ described in the preceding section is not
a solution of the $\SO(3)$-monopole equations, but it is almost a
solution. That is, if we write the $\SO(3)$-monopole equations as a
map to a suitable Banach space,
$$
\fS:\tsC_{\ft} \to \bB,
$$
then the norm $\|\fS(\bga_\Si'(\bA))\|$ is small.  It is possible to
deform the image of the splicing map so that the deformed image
contains a neighborhood of $M_{\fs}\times\Si$ in $\bar\sM_{\ft}$ as
follows.

First, the linearization of the map $\fS$ is not surjective.  The
cokernel of the linearization can be stabilized with a vector
bundle,
$$
\Upsilon_\Si\to\Gl(\ft,\fs,\Si),
$$
in the sense that the fiber of $\Upsilon_\Si$ over
$(A',\Phi')\in\Imag(\bga'_\Si)$ contains the cokernel of the
linearization of $\fS$ at $(A',\Phi')$. There is then, up to gauge
transformation, a unique solution, $\fp_\Si(A',\Phi')$, of the
equation
$$
\fS((A',\Phi')+\fp_\Si(A',\Phi'))\in \Upsilon_\Si|_{(A',\Phi')},
$$
We call $\fp(A',\Phi')$ the {\em gluing perturbation\/} and the
map $(A',\Phi')\mapsto \fp(A',\Phi')$ is smooth. The gluing map is
then defined by
$$
\bga_\Si=\bga_\Si'+ \fp\circ\bga_\Si'.
$$
The {\em obstruction section\/}, $\fo_\Si$, of $\Upsilon_\Si$ is
defined by $\fS\circ\bga_\Si$. The construction of the map $\fp$
appears in \cite{FL3} and it follows that

\begin{thm}
\label{thm:GluingThm}
The restriction of the gluing map $\bga_\Si$
to the zero-locus of the obstruction map, $\fo_\Si^{-1}(0)$,
parameterizes a neighborhood of $M_{\fs}\times \Si$ in
$\bar\sM_{\ft}$.
\end{thm}

\section{Understanding the overlaps}
\label{sec:Overlaps}
When $\ell=1$, Theorem \ref{thm:GluingThm} is
all we need to compute the intersection number
\eqref{eq:ReducibleIntersectionNumber} because there is only one
stratum $X=\Sym^1(X)$. In \cite{FLLevelOne}, we argue that this
intersection number can be written as
$$
\#\left( \bar\sV(z)\cap\bar\sW^{n-1}\cap\bL_{\ft,\fs} \right)
= \langle e(\Upsilon_\Si)\smile\barmu_p(z)\smile\barmu_c^{n-1},
[\rd \Gl(\ft,\fs,X)] \rangle,
$$
where $\Gl(\ft,\fs,X)$ is defined in \eqref{eq:LocalFiberBundle}.
We use our understanding of the structure group of the fibration
\eqref{eq:LocalFiberBundle} and a pushforward formula for the map
$\pi_X$ defined in \eqref{eq:LocalFiberProjection} to compute the
above cohomological pairing.

For $\ell>1$, we cannot use Theorem \ref{thm:GluingThm} and the
information on the local bundle \eqref{eq:LocalFiberBundle}
given in \S \ref{subsec:SplicingMap} to compute the intersection number
\eqref{eq:ReducibleIntersectionNumber} because more than one gluing
map is necessary to cover a neighborhood of
$M_{\fs}\times\Sym^\ell(X)$. One needs to understand the overlap of
these maps to do a cohomological computation. That is, one cannot
just add up the contributions from each open set as there might be
intersection points in the overlap of two or more maps. Such
problems, with only two maps needed, have been addressed in
\cite{OzsvathBlowUp,LenessBlowUp,LenessWC}. This problem becomes
significantly more difficult with three or more gluing maps are
involved.

The difficulty in understanding these overlaps arises largely from
the definition of the gluing perturbation $\fp$.  That is, while
the splicing map $\bga'_\Si$ is quite explicit, the gluing
perturbation $\fp$ arises from an implicit function theorem
argument and is thus not sufficiently explicit for us to compute
$\bga_{\Si}^{-1}\circ\ga_{\Si'}$.

The idea underlying \cite{FL5} is to show that images of the {\em
splicing maps\/}, rather than the images of the gluing maps, satisfy
the conditions of Definition \ref{defn:LocalConeBundles}. As the
splicing maps are presently defined, this approach might seem
problematic because the images of the splicing maps of two different
strata need not intersect at all, even when the images of the
associated gluing maps do. In \cite{FL5}, we introduce deformations
of the domains of the splicing maps and of the maps themselves to
get splicing maps, $\bga_\Si''$ ,whose overlaps are controlled by
push-out diagrams of the form
\begin{equation}
\label{eq:SplicingPushout}
\begin{CD}
\Gl(\ft,\fs,\Si,\Si')@> \rho^d_{\Si,\Si'}>> \Gl(\ft,\fs,\Si)
\\
@V \rho^u_{\Si,\Si'} VV @V\ga_\Si'' VV
\\
\Gl(\ft,\fs,\Si') @> \bga_{\Si'}'' >> \bar\sC_\ft
\end{CD}
\end{equation}
where the maps $\rho^d_{\Si,\Si'}$ and $\rho^u_{\Si,\Si'}$ are open
embeddings and $\bar\sC_{\ft}$ is defined in
\eqref{eq:ConfigComplete}. We describe the deformation of the domain
in \S \ref{subsec:DeformingTheFiber} and we describe the deformation
of the splicing map in \S \ref{subsec:DeformingSpliceMaps}.

\subsection{Deforming the fiber}
\label{subsec:DeformingTheFiber} The deformation of the domain of
the splicing map mentioned above is a deformation of the fiber of
the diagram \eqref{eq:LocalFiberBundle}. We construct the {\em
spliced ends moduli space\/} $\bar M^{s,\diamond}_{spl,\ka}(S^4)$ as
a deformation of the moduli space $\bar M^{s,\diamond}_\ka(S^4)$
described in \S \ref{subsubsection:S4Conn}. The deformation consists
of replacing  neighborhoods of the punctured trivial strata with
images of splicing maps. By punctured trivial strata, we mean the
subspace,
\begin{equation}
\label{eq:TrivialStrata} \{[\Theta]\}\times
\left(\Sym^{\ka,\diamond}(\RR^4)-\{c_\ka\}\right) \subset \bar
M^{s,\diamond}_\ka(S^4),
\end{equation}
where $\Theta$ is the trivial connection,
$\Sym^{\ka,\diamond}(\RR^4)$ is the subspace of the symmetric
product given by points with center-of-mass equal to zero, and
$c_\ka\in\Sym^{\ka,\diamond}(\RR^4)$ is the point supported
entirely at the origin in $\RR^4$.

A neighborhood of the stratum
$\{[\Theta]\}\times\Si$ in $\bar M^{s,\diamond}_\ka(S^4)$ is
parameterized by a gluing map,
\begin{equation}
\label{eq:S4TrivialGluingDomain} \bga_{S^4,\Si}:
\Si\times\prod_{i=1}^r \bar M^{s,\diamond}_{\ka_i}(S^4) \to
\barM^s_{\ka}(S^4),
\end{equation}
where $\ka_i+\dots+\ka_r=\ka$ is the partition of $\ka$
determining the stratum $\ka$. (Note that the domain of
$\bga_{S^4,\Si}$ is actually twisted by a symmetric group
action, but that is not relevant to this discussion.) We
will write
$$
\sN_\ka \subset \bar M^{s,\diamond}_{\ka_i}(S^4)
$$
for the neighborhood of the subspace \eqref{eq:TrivialStrata}
given by the union of the images of the gluing maps
$\bga_{S^4,\Si}$.

Each gluing map $\bga_{S^4,\Si}$ is a deformation of a splicing map
$\bga_{S^4,\Si}'$ defined as in \S \ref{subsec:SplicingMap}. Hence,
there is an isotopy $\Ga_{S^4,\Si}$ with
$\Ga_{S^4,\Si}(1,\cdot)=\bga_{S^4,\Si}(\cdot)$ and
$\Ga_{S^4,\Si}(0,\cdot)=\bga_{S^4,\Si}'(\cdot)$. Make the obvious
extension of $\Ga_{S^4,\Si}(t,\cdot)$ to all $t\in\RR$. We wish to
write
\begin{equation}
\label{eq:SplicingImageSpace} W_\ka=\cup_\Si
\Imag(\bga_{S^4,\Si}'),
\end{equation}
and then to define a collar of $\rd W_\ka$ by a function
$\la:W_\ka\to [0,1]$ such that $\la^{-1}([0,1/2))$ contains
\eqref{eq:TrivialStrata} and $\la^{-1}(1)=\rd W_\ka$. Then, we
wish to replace $\sN_{\ka}$ with a deformation of $W_\ka$, defined
by replacing a point $\bga_{S^4,\Si}(\bA)$ with
$\Ga_{S^4,\Si}(2\la(\bA)-1,\bA)$.

The problem with the above description is that the space $W_\ka$
is not homeomorphic to $\sN_\ka$ because the images of the
splicing maps defined by different strata $\Si$ need not
intersect. The deformation of $W_\ka$ thus need not be an isotopy
and the resulting space need not be smoothly-stratified.

We repair this argument inductively. We begin by observing that the
splicing map has an associative property which allows us, assuming
the connections in the domain of $\bga_{S^4,\Si}'$ are themselves
the image of a splicing map, to control the overlaps of different
splicing maps. This control then makes $W_\ka$ a smoothly stratified
space homeomorphic to $\sN_{\ka}$, competing the naive argument
above.

\subsubsection{Associativity of the splicing map}
Splicing connections on $S^4$ to the trivial connection on $S^4$
has a nice property which we refer to as the {\em associativity of
splicing\/}.  It can be summarized as follows. Let $A_{i,j}$ be
connections on bundles $E_{\ka_{i,j}}\to S^4$. Let $x_{i,j}$ and
$y_i$ be distinct points in $\RR^4$. Define
$$
A_i=\Theta\#_{x_{i,1}}A_{i,1}\#\dots\#_{x_{i,r_i}}A_{i,s_i}
$$
to be the connection obtained by splicing the connections $A_{i,j}$
to the trivial connection $\Theta$ at the points $x_{i,j}$. Then, if
the connections $A_{i,j}$ are sufficiently concentrated near the
north pole in the sense of \eqref{eq:DilationParameter} relative to
the separation of the points $x_{i,j}$, one has
\begin{equation}
\label{eq:AssociativityOfSplicing}
\begin{aligned}
{}&\Theta\#_{y_1}A_1\#\dots \#_{y_r}A_r
\\
{}&\quad= \Theta\#_{y_1+x_{1,1}}A_{1,1}\#\dots
\#_{y_i+x_{i,j}}A_{i,j}\#\dots \#_{y_r+x_{r,s_r}}A_{r,s_r}
\end{aligned}
\end{equation}
Equation \eqref{eq:AssociativityOfSplicing} identifies the
composition of two splicing maps with a single splicing map. This
identity is crucial in the construction described in the following
section. We note that it is our ability to write down the splicing
map explicitly that makes it possible to obtain Equation
\eqref{eq:AssociativityOfSplicing}.

\subsubsection{Overlapping splicing maps}
We now define  the spliced-end moduli space,
$\barM^{s,\diamond}_{spl,\ka}(S^4)$. This moduli space is a
deformation of $\barM^{s,\diamond}_{\ka}(S^4)$ with the property
that a neighborhood of the subspace \eqref{eq:TrivialStrata} is
given by the image of splicing maps instead of the gluing maps in
\eqref{eq:S4TrivialGluingDomain}. This construction uses induction
on $\ka$. For $\ka=1$, the trivial strata \eqref{eq:TrivialStrata}
is empty because of the absence of the cone point $c_1$ from
\eqref{eq:TrivialStrata}, and we may define
$$
\barM^{s,\diamond}_{spl,1}(S^4)=\barM^{s,\diamond}_{1}(S^4).
$$
Because the cone point, $c_\ka$, is not included in the strata in
\eqref{eq:TrivialStrata}, the domains of the gluing maps
\eqref{eq:S4TrivialGluingDomain} contain only moduli spaces
$\barM^{s,\diamond}_{\ka_i}(S^4)$ with $\ka_i<\ka$.  Thus, using
induction we may require the neighborhood of the punctured trivial
strata to be parameterized by the splicing map with domain
\begin{equation}
\label{eq:ModifiedSplicingDomain}
\Si\times\prod_{i=1}^r \barM^{s,\diamond}_{spl,\ka_i}(S^4),
\end{equation}
instead of the domain of the gluing map
\eqref{eq:S4TrivialGluingDomain}. The associativity of splicing to
the trivial connection \eqref{eq:S4TrivialGluingDomain}, then
ensures that the overlap of the images of two such splicing maps can
be understood by the following pushout diagram:
\begin{equation}
\label{eq:TrivialOverlap}
\begin{CD}
\nu(\Si,\Si')\times
\prod_{i,j}\barM^{s,\diamond}_{spl,\ka_{i,j}}(S^4) @>
\rho^u_{\Si,\Si'}(S^4) >> \Si'
\times\prod_{i,j}\barM^{s,\diamond}_{spl,\ka_{i,j}}(S^4)
\\
@V \rho^d_{\Si,\Si'}(S^4) VV @V \bga_{\Si'}(S^4)' VV
\\
\Si \times\prod_i\barM^{s,\diamond}_{spl,\ka_i}(S^4) @>
\bga_\Si(S^4)'>> \bar \sB
\end{CD}
\end{equation}
We now explain the diagram \eqref{eq:TrivialOverlap}. Let
$\Si<\Si'$ be strata of the subspace \eqref{eq:TrivialStrata}. The
space $\bar\sB$ in the diagram \eqref{eq:TrivialOverlap} is
defined analogously to the space $\bar\sC_{\ft}$ defined in
\eqref{eq:ConfigComplete}.

We now explain the maps $\rho^u_{\Si,\Si'}(S^4)$ and
$\rho^d_{\Si,\Si'}(S^4)$ appearing in \eqref{eq:TrivialOverlap}.
The upper stratum, $\Si'$, is given by a refinement of the
partition giving the lower stratum, $\Si$. That is if $\Si<\Si'$
and $\Si$ is given by the partition $\ka_1+\dots+\ka_r$, then
$\Si'$ will be given by a partition $\ka=\sum_{i,j}\ka_{i,j}$
where $\ka_i=\sum_j\ka_{i,j}$. (There could be more than one such
refinement. Each such refinement corresponds to a different
component of the end of $\Si'$ near $\Si$ and can be treated
separately.  For simplicity of exposition, we ignore both this
issue and problems involving the symmetric group action here.) Let
$\nu(\Si,\Si')\subset\Si'$ be the intersection of a tubular
neighborhood of $\Si$ in $\Sym^{\ka,\diamond}(\RR^4)$ with $\Si'$. The map
$\rho^u_{\Si,\Si'}(S^4)$ in the diagram \eqref{eq:TrivialOverlap}
is defined by the inclusion $\nu(\Si,\Si')\subset\Si'$.

Let $p_{\Si',\Si}:\nu(\Si,\Si')\to \Si$ be the tubular
neighborhood projection map. The fiber of $p_{\Si',\Si}$ will be
points in $\RR^4$. Splice the connections given by a point in
$\prod_{j}\barM^{s,\diamond}_{spl,\ka_{i,j}}(S^4)$ to the trivial
connection at the points in $\RR^4$ given by the point in the
fiber of $p_{\Si',\Si}$. The resulting connection lies in $\bar
M^{s,\diamond}_{spl,\ka_i}(S^4)$ because of the inductive
hypothesis that a neighborhood of the punctured trivial strata in
$\bar M^{s,\diamond}_{spl,\ka_i}(S^4)$ lies in the image of the
splicing map with domain \eqref{eq:ModifiedSplicingDomain}.
The projection map $p_{\Si',\Si}$ and this splicing
construction then define the map $\rho^d_{\Si,\Si'}(S^4)$ in the
diagram \eqref{eq:TrivialOverlap}.

Then, the associativity of splicing to the trivial connection
\eqref{eq:AssociativityOfSplicing} implies that the diagram
\eqref{eq:TrivialOverlap} commutes. Specifically, the compositions
$$
\bga_\Si(S^4)'\circ\rho^d_{\Si,\Si'}(S^4) \quad\text{and}\quad
\bga_{\Si'}(S^4)'\circ\rho^u_{\Si,\Si'}(S^4)
$$
are the splicing maps appearing on the left-hand-side and
right-hand-side, respectively, of
\eqref{eq:AssociativityOfSplicing}. Moreover, one can also show that
any point in the images of both $\bga_{\Si}'$ and of $\bga_{\Si'}'$
appears in the pushout \eqref{eq:TrivialOverlap}. Hence, the
overlaps of the images of the splicing maps are then controlled by
the pushout diagram \eqref{eq:TrivialOverlap}.

Because the maps $\rho^u_{\Si,\Si'}(S^4)$ and
$\rho^d_{\Si,\Si'}(S^4)$ are open embeddings, the union of the
images of the splicing maps, $W_\ka$, is a smoothly-stratified
space. The gluing perturbation gives a smoothly-stratified isotopy
between $W_\ka$ and $\sN_{\ka}$. The argument given before
equation \eqref{eq:SplicingImageSpace} then shows how to deform a
collar of the boundary of $W_\ka(S^4)$ into
$\barM^{s,\diamond}_\ka(S^4)$. The resulting space is
$\barM^{s,\diamond}_{spl,\ka}(S^4)$, completing the induction.

\subsection{Deforming the splicing map}
\label{subsec:DeformingSpliceMaps} Let $\Si<\Si'$ be strata of
$\Sym^\ell(X)$. With the domain, $\Gl(\ft,\fs,\Si)$, of the splicing
maps $\bga_\Si'$ redefined by replacing the spaces
$\barM^{s,\diamond}_\ka(S^4)$ appearing in the fiber of
\eqref{eq:LocalFiberBundle} with the spaces
$\barM^{s,\diamond}_{spl,\ka}(S^4)$, we now describe the
deformations of the splicing maps
$$
\bga_\Si'':\Gl(\ft,\fs,\Si) \to \bar\sC_{\ft},\quad
\bga_{\Si'}'':\Gl(\ft,\fs,\Si') \to \bar\sC_{\ft}
$$
needed to ensure that the overlap of their images is controlled by a
pushout diagram of the form \eqref{eq:SplicingPushout}.

Let $\nu(\Si,\Si')$ again denote the intersection of $\Si'$ with a
tubular neighborhood of $\Si$. We define the space
$\Gl(\ft,\fs,\Si,\Si')$ to be the restriction of the fiber bundle
$\Gl(\ft,\fs,\Si')$ to $N_{\ft(\ell),\fs}\times\nu(\Si,\Si')$
appearing in \eqref{eq:TrivialOverlap} and we define
$\rho^u_{\Si,\Si'}$ to be the inclusion of bundles:
\begin{equation}
\label{eq:UpwardsOverlap}
\begin{CD}
\Gl(\ft,\fs,\Si,\Si') @> \rho^u_{\Si,\Si'} >> \Gl(\ft,\fs,\Si')
\\
@V \Pi_{\Si'} VV  @V \Pi_{\Si'} VV
\\
N_{\ft(\ell),\fs}\times\nu(\Si,\Si') @>>>
N_{\ft(\ell),\fs}\times\Si'
\end{CD}
\end{equation}
The definition of
$$
\rho^d_{\Si,\Si'}: \Gl(\ft,\fs,\Si,\Si')\to \Gl(\ft,\fs,\Si),
$$
is similar to that of the map $\rho^d_{\Si,\Si'}(S^4)$ appearing
in \eqref{eq:TrivialOverlap}. Let $p_{\Si,\Si'}:\nu(\Si,\Si')\to
\Si$ be the projection map of the tubular neighborhood. The fiber
of $p_{\Si,\Si'}$ is, up to a choice of a trivialization of the
tangent bundle of $X$, a collection of points in $\RR^4$. The
fiber of the composition,
$$
(\id_{N_{\ft(\ell),\fs}}\times p_{\Si,\Si'})\circ \Pi_{\Si'}:
\Gl(\ft,\fs,\Si,\Si') \to N_{\ft(\ell),\fs}\times\Si
$$
is then a collection of points in $\RR^4$ and the framed connections
in the fiber of $\Pi_{\Si'}$. Just as in the construction of the map
$\rho^d_{\Si,\Si'}(S^4)$, this data can be spliced to the trivial
connection to get connections on $S^4$, giving an element of the
fiber of $\Pi_\Si$. (Note that for the definition of the map
$\rho^d_{\Si,\Si'}$ to make sense, we must redefine the fibers of
$\Pi_\Si$ in $\Gl(\ft,\fs,\Si)$ to have the connections in $W_\ka$
rather than $\sN_{\ka}$, as discussed in the beginning of \S
\ref{subsec:DeformingTheFiber}.) This defines
$\rho^d_{\Si,\Si'}$ as a fiber bundle map:
$$
\begin{CD}
\Gl(\ft,\fs,\Si,\Si') @> \rho^d_{\Si,\Si'} >> \Gl(\ft,\fs,\Si)
\\
@V\Pi_{\Si'} VV @V\Pi_\Si VV
\\
N_{\ft(\ell),\fs} \times \nu(\Si,\Si') @>
\id_{N_{\ft(\ell),\fs}}\times p_{\Si,\Si'} >>
N_{\ft(\ell),\fs}\times \Si
\end{CD}
$$
Then, we wish to define deformations, $\bga_\Si''$ and
$\bga_{\Si'}''$, of the splicing maps $\bga_\Si'$ and
$\bga_{\Si'}'$, so that the following diagram commutes
\begin{equation}
\label{eq:SplicingPushout2}
\begin{CD}
\Gl(\ft,\fs,\Si,\Si')@> \rho^d_{\Si,\Si'}>> \Gl(\ft,\fs,\Si)
\\
@V \rho^u_{\Si,\Si'} VV @V\ga_\Si'' VV
\\
\Gl(\ft,\fs,\Si') @> \bga_{\Si'}'' >> \bar\sC_\ft
\end{CD}
\end{equation}
The corresponding diagram for splicing to the trivial connection on
$S^4$, \eqref{eq:TrivialOverlap}, commuted because of the
associativity of splicing equality,
\eqref{eq:AssociativityOfSplicing}. No such equality holds when the
connection to which one is splicing, $A_0$ in the language of \S
\ref{subsec:SplicingMap}, is not flat and the manifold to which
one is splicing does not admit a flat metric. However, one
can ``flatten'' the connection $A_0$ and the metric $g$ on $X$ on
small balls around the splicing points in $\bx$.
To obtain the deformed splicing map,
$\bga_\Si''$, one then simply ``flattens'' the connection $A_0$ on
larger neighborhoods of the splicing points, using a locally
flattened metric to identify neighborhoods of these points with
neighborhoods of the north pole in $S^4$.  With this deformation of
the splicing map, the diagram \eqref{eq:SplicingPushout2} commutes.

Define
\begin{equation}
\label{eq:GlobalSplicingData}
\Gl(\ft,\fs,X)=\cup_{\Si\subset\Sym^\ell(X)}\bga_{\Si}''(\Gl(\ft,\fs,\Si)).
\end{equation}
The space $\Gl(\ft,\fs,X)$ is a smoothly-stratified space given
by a union of local cone bundle neighborhoods.
The pushout diagram \eqref{eq:SplicingPushout2} controls
the overlaps of these diagrams and from this control,
one can see that:

\begin{thm}
\label{thm:ConesControlled}
The space $\Gl(\ft,\fs,X)$ is a union of local cone bundle
neighborhoods which satisfy the conditions of
Definition \ref{defn:LocalConeBundles}.
\end{thm}

In \cite{FL4}, we will extend the results of \cite{FL3}
constructing a gluing
perturbation of the image of these splicing maps to parameterize a
neighborhood of $M_{\fs}\times\Sym^\ell(X)$ in $\bar\sM_{\ft}/S^1$,
giving the following technical result on which the proofs
of Theorem \ref{conj:PTConj} and Theorem \ref{thm:CobordismResult} rely.

\begin{thm}
\label{thm:ExtendedGluingThm}
There is a section $\fo$ of a pseudo-vector
bundle $\Upsilon\to \Gl(\ft,\fs,X)$ and a stratum-preserving
deformation of the inclusion $\Gl(\ft,\fs,X)\to\bar\sC_{\ft}/S^1$
such that the restriction of this deformation to $\fo^{-1}(0)$
parameterizes a neighborhood of $M_{\fs}\times\Sym^\ell(X)$ in
$\bar\sM_{\ft}/S^1$. In addition, the restriction of the
obstruction section $\fo$ to any stratum vanishes transversely.
\end{thm}

\begin{rmk}
The obstruction bundle, $\Upsilon\to \Gl(\ft,\fs,X)$, is only a
pseudo-vector bundle because its rank depends on the stratum of
$\Gl(\ft,\fs,X)$. Although $\Upsilon$ is not a vector bundle,
using a relative Euler class argument, one can show that there is
a rational cohomology class on $\rd\Gl(\ft,\fs,X)$ which acts as
an Euler class of $\Upsilon$.  We refer to this class as
$e(\Upsilon)$.
\end{rmk}

We omit any further discussion of the gluing perturbation because
as it is a deformation, to do any cohomological calculations, it
suffices to work with the images of the deformed splicing maps.

Finally, we observe that there is an $S^1$ action on
the space $\Gl(\ft,\fs,X)$ and on the obstruction bundle $\Upsilon$
such that the splicing map and gluing perturbation are equivariant
with respect to this action on $\Gl(\ft,\fs,X)$ and the action
\eqref{eq:S1Action} on $\bar\sC_{\ft}$.

\section{The cohomological formalism}
\label{sec:CohomolFormalism} Using Theorem
\ref{thm:ExtendedGluingThm}, the desired intersection number can be
written as a cohomological pairing,
\begin{equation}
\label{eq:DualPairing}
\#\left(
\bar\sV(z)\cap\bar\sW^{n-1}\cap\bL_{\ft,\fs} \right)
=
\langle e(\Upsilon)\smile\barmu_p(z)\smile\barmu_c^{n-1}, [\rd \Gl(\ft,\fs,X)/S^1] \rangle,
\end{equation}
to which Theorems \ref{thm:ConesControlled} and
\ref{thm:ExtendedGluingThm} allow us to apply the formalism of \S
\ref{subsec:Model}.  Explicitly, we write
$$
\rd
\Gl(\ft,\fs,X)/S^1=
\cup_i\ \bL^{\vir}_{\ft,\fs}(\Si_i),
$$
where $t_i:\Gl(\ft,\fs,\Si_i)\to [0,\infty)$ is the cone function
and
$$
\bL^{\vir}_{\ft,\fs}(\Si_i)=
\Gl(\ft,\fs,\Si_i)/S^1\cap t_i^{-1}(\eps_i)
-\left(\cup_{j\neq i} t_j^{-1}([0,\eps_j))\right).
$$
The compatible structure group condition implies that
there is a pre-compact subspace $K_i\Subset \Si_i$
such that
the restriction of $\pi_i$ to $\bL^{\vir}_{\ft,\fs}(\Si_i)$ defines
a fiber bundle,
\begin{equation}
\label{eq:LinkPieceBundle} \pi_i: \bL^{\vir}_{\ft,\fs}(\Si_i)\to
M_{\fs}\times K_i,
\end{equation}
with the same
structure group as the fiber bundle in \eqref{eq:LocalFiberBundle}.
Hence, the characteristic classes of the bundle
\eqref{eq:LinkPieceBundle} are given by appropriate symmetric
products of $p_1(X)$, $e(X)$, $c_1(\fs)$, $c_1(\ft)$, and
$p_1(\ft)$, and the cohomology class
$$
\mu_\fs\in H^2(M_{\fs})
$$
defining the Seiberg-Witten invariant.

As described after \eqref{eq:DecomposeL}, the strata of
$\bL^{\vir}_{\ft,\fs}(\Si_i)$ are smooth manifolds with corners. The
boundary of $\bL^{\vir}_{\ft,\fs}(\Si_i)$ is
\begin{equation}
\label{eq:DefineLinkPieceBoundary}
\rd \bL^{\vir}_{\ft,\fs}(\Si_i)
=
\cup_{j\neq i}\ \rd_j \bL^{\vir}_{\ft,\fs}(\Si_i)
\quad\text{where}\quad
\rd_j \bL^{\vir}_{\ft,\fs}(\Si_i)= \bL^{\vir}_{\ft,\fs}(\Si_i)\cap t_j^{-1}(\eps_j).
\end{equation}
We abbreviate the cohomology class in \eqref{eq:DualPairing} by
$$
\Om(z,n)=
e(\Upsilon)\smile\barmu_p(z)\smile\barmu_c^{n-1}.
$$
To apply the pushforward-pullback argument in \S \ref{subsec:Model}
to compute the pairing \eqref{eq:DualPairing}, we need to select a
representative of the cohomology class $\Om(z,n)$ such that:
\begin{itemize}
\item
The representative has compact support away from the boundaries
$\rd \bL^{\vir}_{\ft,\fs}(\Si_i)$,
\item
The restriction of the representative
to each component $\bL^{\vir}_{\ft,\fs}(\Si_i)$ is a product
of an equivariant cohomology class on the fiber and a
cohomology class pulled back from the base of the fiber
bundle \eqref{eq:LinkPieceBundle}.
\end{itemize}
We specify such a representative of $\Om(z,n)$ by constructing a
quotient, $\widehat\bL^{\vir}_{\ft,\fs}$, of $\rd\Gl(\ft,\fs,X)/S^1$
from which $\Om(z,n)$ is pulled back. Recall that
$z=h^{\delta-2m}x^m$, where $h\in H_2(X;\RR)$ and $x\in H_0(X;\ZZ)$
was the generator.

\begin{prop}
\label{prop:QuotientProperties}
There is a quotient $\widehat \bL^{\vir}_{\ft,\fs}$
of $\rd\Gl(\ft,\fs,X)/S^1$ with quotient map
$$
q:\rd\Gl(\ft,\fs,X)/S^1 \to \widehat \bL^{\vir}_{\ft,\fs},
$$
such that
\begin{equation}
\label{eq:CohomPulledBack}
\Om(z,n)=q^*\widehat\Om(z,n),\quad\text{for}\quad
\widehat\Om(z,n)\in H^d(\widehat\bL^{\vir}_{\ft,\fs}),
\end{equation}
where $d=\dim \rd\Gl(\ft,\fs,X)/S^1$, which has the following
properties:
\begin{enumerate}
\item
The image of each boundary, $q(\rd_j\bL^{\vir}_{\ft,\fs}(\Si_i))$,
has codimension greater than or equal to two,
\item
The image of each component under the quotient map,
$$
\widehat \bL^{\vir}_{\ft,\fs}(\Si_i)=q(\bL^{\vir}_{\ft,\fs}(\Si_i)),
$$
is a fiber bundle fitting into the diagram
\begin{equation}
\label{eq:QuotientFiberDiagram}
\begin{CD}
\bL^{\vir}_{\ft,\fs}(\Si_i)
@> q >>
\widehat \bL^{\vir}_{\ft,\fs}(\Si_i)
@> \tilde g_i >>
\EG_i\times_{G_i}\widehat F_i(\beps)
\\
@V \pi_i VV  @V \hat\pi_i VV @V m_iVV
\\
M_{\fs}\times  K_i @>>>
M_{\fs}\times \cl(\Si_i)
@> g_i >> \BG_i
\end{CD}
\end{equation}
where $\widehat F_i(\beps)$ is the fiber of $\hat\pi_i$.
\item
The homotopy type of the fiber bundle
$\hat\pi_i$ in \eqref{eq:QuotientFiberDiagram}
depends only on the characteristic classes
of the bundle \eqref{eq:LinkPieceBundle}.
\item
The restriction of $\widehat\Om(z,n)$ to $\widehat \bL^{\vir}_{\ft,\fs}(\Si_i)$,
$\widehat \Om_i$, satisfies
\begin{equation}
\label{eq:QuotientCohom} \widehat\Om_i= \sum_j p_{i,j}(\tilde
g_i^*\hat \nu^{d-j})\smile (\widehat\pi_i^*\hat\om_{i,j}),
\end{equation}
where $\hat\nu$ is the first Chern class of an $S^1$ action on the
fibers $\widehat F_i(\beps)$, $\hat \om_{i,j}\in H^j(M_{\fs}\times
\widehat K_i)$ is a polynomial in
characteristic classes of $\pi_i$ and the Poincar\'e duals of $h$
and $x$, and $p_{i,j}$ are universal constants.
\end{enumerate}
\end{prop}

\begin{proof}
We now sketch
the construction of the quotient $\widehat\bL^{\vir}_{\ft,\fs}$.
The quotient is defined by ``collapsing'' the boundaries of the
components of the link defined in \eqref{eq:DefineLinkPieceBoundary}.
From the compatible structure group conditions, we know that
for $i<j$, the boundary $\rd_j\bL^{\vir}_{\ft,\fs}(\Si)$
is a subbundle of the fiber bundle \eqref{eq:LinkPieceBundle}.
From the Thom-Mather control conditions, we know that
for $j<i$, the boundary $\rd_j\bL^{\vir}_{\ft,\fs}(\Si)$ is
the restriction of the bundle \eqref{eq:LinkPieceBundle} to
a boundary $M_{\fs}\times\rd_jK_i$ of the base.

We define the quotient of the boundary $\rd_j\bL^{\vir}_{\ft,\fs}(\Si)$
for $j<i$. First, we observe that $\cl(\Si_i)$ can be presented as a quotient of
$K_i$. Specifically, one notes that the
union of singular strata in $\cl(\Si_i)$, $\cl(\Si_i)-\Si_i$, is
a neighborhood deformation retraction in $\cl(\Si_i)$.  The restriction
of this retraction to $K_i\subset \Si_i$ can be used to define a surjective map
$K_i\to\cl(\Si_i)$ which we consider as a quotient map.
The fiber bundle \eqref{eq:LinkPieceBundle} extends
over $M_{\fs}\times\cl(\Si_i)$ and this extension can  be presented
as a quotient of $\bL^{\vir}_{\ft,\fs}(\Si)$ which
collapses the boundary $\rd_j\bL^{\vir}_{\ft,\fs}(\Si)$
for $j<i$ to the restriction of the extended fiber bundle over the lower
stratum  $\Si_j\subset\cl(\Si_i)$.

The definition of the boundaries of $\bL^{\vir}_{\ft,\fs}(\Si_i)$
in \eqref{eq:DefineLinkPieceBoundary} implies that
$$
\rd_j\bL^{\vir}_{\ft,\fs}(\Si_i)=\rd_i\bL^{\vir}_{\ft,\fs}(\Si_j).
$$
Thus, for $i<j$ we must take the quotient of the boundary
$\rd_j\bL^{\vir}_{\ft,\fs}(\Si_i)$ exactly as we have done for the
boundary $\rd_i\bL^{\vir}_{\ft,\fs}(\Si_j)$. Because the boundary
$\rd_j\bL^{\vir}_{\ft,\fs}(\Si_i)$ is a subbundle of the bundle
\eqref{eq:LinkPieceBundle}, we can then use the property of
compatible structure groups to argue that this quotient can be
obtained by taking a quotient of the fiber without changing the
structure group.  On the intersections,
$$
\rd_{j_1}\rd_{j_2}\dots \rd_{j_r}\bL^{\vir}_{\ft,\fs}(\Si_i)
=
\cap_{k=1}^r \rd_{j_k} \bL^{\vir}_{\ft,\fs}(\Si_i)
 \quad\text{for $i<j_1<\dots< j_r$},
$$
there are $r$ quotients, defined by the inclusions
$$
\rd_{j_1}\rd_{j_2}\dots \rd_{j_r}\rd_k\bL^{\vir}_{\ft,\fs}(\Si_i)
\subset
\rd_{j_k}\bL^{\vir}_{\ft,\fs}(\Si_i)
$$
and the quotients on $\rd_{j_k}\bL^{\vir}_{\ft,\fs}(\Si_i)$.
One must verify that these multiple quotients are well-defined
and respect the structure group.

The properties of the cohomology
class $\Om(z,n)$ are easily verified.
\end{proof}

Finally, we discuss how Theorem \ref{conj:PTConj} follows from
\ref{prop:QuotientProperties}. The first condition in Proposition
\ref{prop:QuotientProperties} yields the identity
\begin{equation}
\label{eq:DecomposeQuotientFund}
q_*[\rd\Gl(\ft,\fs,X)/S^1]
=
\sum_i [\widehat\bL^{\vir}_{\ft,\fs}(\Si_i)].
\end{equation}
The third property in Proposition \ref{prop:QuotientProperties}
and the pushforward-pullback formula
imply that
\begin{equation}
\label{eq:PushPullEquality}
(\hat\pi_i)_*(\tilde g_i^*\nu^k)
=g_i^*(m_{i*}\hat\nu^k)
\end{equation}
is given by a universal polynomial
in the characteristic classes of $\hat\pi_i$. We
then compute the intersection number in \eqref{eq:DualPairing} by
\begin{equation}
\label{eq:PairingComp}
\begin{aligned}
{}&
\langle\Om(z,n), [\rd\Gl(\ft,\fs,X)/S^1]\rangle
\\
{}&\quad=
\langle q^*\widehat\Om(z,n) ,[\rd\Gl(\ft,\fs,X)/S^1]\rangle
\quad\text{by \eqref{eq:CohomPulledBack}}
\\
{}&\quad=
\langle \widehat\Om(z,n),q_*[\rd\Gl(\ft,\fs,X)/S^1]\rangle
\\
{}&\quad=
\sum_i \langle \widehat\Om_i,[\widehat\bL^{\vir}_{\ft,\fs}(\Si_i)]\rangle
\quad\text{by \eqref{eq:DecomposeQuotientFund}}
\\
{}&\quad=
\sum_{i,j}
p_{i,j}\langle \tilde g_i^*\hat \nu^{d-j}\smile \hat\pi_i^*\hat \om_{i,j},
          [\widehat\bL^{\vir}_{\ft,\fs}(\Si_i)]\rangle
\quad\text{by \eqref{eq:QuotientCohom}}
\\
{}&\quad=
\sum_{i,j}
p_{i,j}\langle
(\hat\pi_i)_*( \tilde g_i^*\hat \nu^{d-j})\smile  \hat \om_{i,j},
[M_{\fs}\times \cl(\Si_i)]\rangle
\\
{}&\quad=
\sum_{i,j}
p_{i,j}\langle
g_i^*(m_{i*}\hat \nu^{d-j})\smile \hat \om_{i,j},
[M_{\fs}\times \cl(\Si_i)]
\rangle
\quad\text{by \eqref{eq:PushPullEquality}}
\end{aligned}
\end{equation}
By the characterization of $g_i^*(m_{i*}\hat \nu^{d-j})$ following
\eqref{eq:PushPullEquality}, we then see that the final expression
is the desired universal polynomial in the characteristic
classes appearing in Theorem \ref{conj:PTConj}.

\bibliographystyle{plain}

\end{document}